\documentclass[11pt]{article}

\usepackage{graphicx,amsmath,amssymb,amsthm,pdfsync}
\usepackage[top=4cm, bottom=3cm, left=3cm, right=3cm]{geometry}

\newtheorem{theorem}{Theorem}[section]
\newtheorem{lemma}[theorem]{Lemma}

\newtheorem{definition}[theorem]{Definition}

\newtheorem{remark}[theorem]{Remark}

\numberwithin{equation}{section}

\newcommand{\RR}{\mathbb{R}}
\newcommand{\resc}{{\bf resc}}
\newcommand{\AC}{\mathrm{ac}}
\newcommand{\mF}{{\mathcal F}}

\newcommand{\pv}{{\rm p.v.\, }}
\newcommand{\sign}{{\rm sign }}
\def\derpar#1#2{\frac{\partial#1}{\partial#2}}

\begin{document}

\title{Refined Asymptotics for the subcritical Keller-Segel system and Related Functional Inequalities}


\author{Vincent Calvez\protect
\footnote{
\'Ecole Normale Sup\'erieure de
Lyon, UMR CNRS 5669 "Unit\'e de Math\'ematiques Pures et
Appliqu\'ees", 46 all\'ee d'Italie, F-69364 Lyon Cedex 07,
France. \texttt{vincent.calvez@ens-lyon.fr}}, Jos\'e Antonio Carrillo\protect
\footnote{Instituci\'o Catalana de Recerca i Estudis Avan\c cats and Departament de
Matem\`atiques, Universitat Aut\`onoma de Barcelona, E-08193
Bellaterra, Spain. \texttt{carrillo@mat.uab.es}}}



\maketitle

\begin{abstract}
We analyze the rate of convergence towards self-similarity for the
subcritical Keller-Segel system in the radially symmetric
two-dimensional case and in the corresponding one-dimensional case
for logarithmic interaction. We measure convergence in Wasserstein
distance. The rate of convergence towards self-similarity does not
degenerate as we approach the critical case. As a byproduct, we
obtain a proof of the logarithmic Hardy-Littlewood-Sobolev
inequality in the one dimensional and radially symmetric two
dimensional case based on optimal transport arguments. In addition
we prove that the one-dimensional equation is a contraction with
respect to Fourier distance in the subcritical case.
\end{abstract}

\section{Introduction}
We will concentrate on seeking decay rates towards equilibria or
self-similarity profiles for aggregation equations with linear
diffusion in the fair competition regime. These models describe
the evolution of a population of individuals which are diffusing
by standard Brownian motion and attracting each other by a
pairwise symmetric potential $W(x)$. We focus on a logarithmic interaction potential
$W(x)=2\chi\log|x|$, with $\chi>0$. The Fokker-Planck equation
governing the evolution of the probability density function
$\rho(t,x)$ associated to this particle system reads as
\begin{equation}\label{eq:KS}
\derpar{\rho}{t} = \nabla\cdot \left [ \dfrac1N \nabla\rho + 2\chi\rho
\left(\nabla \log|x|\ast \rho \right) \right ]\,,\quad t>0\, , \quad x\in \RR^N\, .
\end{equation}
Due to translational invariance and mass conservation, in the rest
of this work we restrict to zero center of mass probability
densities,
\[
\rho(t,x) \geq 0\,,\quad \int_{\RR^N} \rho(t,x)\,dx = 1 \,
,\quad  \int_{\RR^N} x \rho(t,x)\,dx = 0 \, ,
\]
By fair competition, we mean that the dynamics of \eqref{eq:KS}
are driven by a simple dichotomy as in the classical Keller-Segel
system in two dimensions \cite{Keller-Segel-70,patlak,DP,BDP}, the
modified Keller-Segel system in one dimension \cite{CPS,BCC} or
the Keller-Segel model with suitable nonlinear diffusion in larger
dimensions \cite{BCL}. In all these examples there is a critical
parameter which makes the distinction between global existence of
solutions and finite-time blow-up. More precisely, we will discuss
the modified one-dimensional Keller-Segel equation \cite{CPS,BCC}:
\begin{equation}
\partial_t \rho   =  \partial^2_{xx}\rho  +  2\chi \partial_x  \left( \rho \partial_x \left( \log |x|  * \rho \right) \right)\, , \quad t>0\, , \quad x\in \RR\, ,
\label{eq:KS1D}
\end{equation}
and the radially symmetric two-dimensional classical Keller-Segel
equation:
\begin{equation} \label{eq:radialKS}
\partial_t(r\rho(t,r)) =  \dfrac12 \partial_r(r\partial_r \rho(t,r)) +  2\chi\partial_r\left[ \rho(t,r) M[\rho(t)](r)  \right]\, , \quad t>0\, , \quad r\in \RR_+\,,
\end{equation}
where $M[\rho]$ denotes the cumulated mass of
$\rho$ inside balls,
\begin{equation*}
M[\rho](r) = 2\pi \int_{0}^{r}  \rho(s) s\,ds\, .
\end{equation*}
Both equations \eqref{eq:KS1D} and \eqref{eq:radialKS} exhibit a
transition depending on the sensitivity coefficient $\chi$:
\begin{itemize}
\item {\bf Subcritical Case.-} For any $0<\chi<1$
solutions exist globally-in-time and they approach a unique
self-similar solution as $t\to \infty$, see \cite{DP,BDP,Biler1,CPS}.

\item {\bf Critical Case.-} For $\chi=1$ solutions exist
globally-in-time. There are infinitely many stationary solutions
with infinite second moment. Solutions having finite initial
second moment concentrate in infinite time towards the Dirac mass
$\delta_0$ \cite{BCM,Biler1,Souplet}. Solutions of infinite initial second moment
close enough to a stationary solution converge to it as $t\to
\infty$ \cite{BCC-Critical}.

\item {\bf Supercritical Case.-} For any $\chi>1$ smooth
fast-decaying solutions do not exist globally in time \cite{Nagai,Biler2,BDP,CPS}.
\end{itemize}
The critical parameter $\chi = 1$ can be obtained from two formal
computations at this stage. The evolution of the second moment
satisfies in both cases the relation:
\begin{equation*}\label{eq:secondmoment}
\frac12 \frac{d}{dt} \int_\RR |x|^2 \rho(t,x)\,dx = 1-\chi\, .
\end{equation*}
This implies that for $\chi>1$ solutions will necessarily blow-up
before the second moment touches zero. On the other hand, the
Keller-Segel equation  \eqref{eq:KS} is equipped with a free
energy (entropy minus potential energy),
\begin{equation}\label{eq:functional}
\mathcal F[\rho] = \dfrac 1N  \int_{\RR^N}   \rho(x)\log \rho(x)
\, dx + \chi  \iint_{\RR^N\times\RR^N} \rho(x)  \log(|x-y|)
\rho(y)\, dxdy\, .
\end{equation}
It is formally decreasing along the trajectories
\begin{equation} \label{eq:decreasing energy}
\dfrac{d}{dt}\mathcal{F}[\rho(t)] = -\int_{\RR^N} \rho(t,x)\left|
\nabla \left( \dfrac 1{N}\log \rho(t,x)  + 2\chi \log|x|*\rho(t,x)
\right) \right|^2\,dx\,  .
\end{equation}
Moreover, it was shown in \cite{DP,BDP} that for $\chi<1$ the free
energy estimate from above implies an {\em a priori} bound in the entropy part of
the functional which is at the basis of the construction of
global-in-time solutions. This was achieved by using the
Logarithmic-HLS inequality \cite{Be,CarLos} which relates the
entropy and the interaction part of the functional.

Nontrivial equilibrium profiles or critical profiles, only exist
for the critical parameter $\chi = 1$. They are solutions to the
following Euler-Lagrange equations:
\begin{align}
\mu'(x) + 2 \mu(x)\partial_x\left(\log|x|*\mu(x)\right) = 0\, ,
\label{eq:equality1D} \\
\dfrac12 r \mu'(r)  +  2  \mu(r) M[\mu](r) = 0
\label{eq:equality2D}  \, ,
\end{align}
resp. in dimension $N = 1$ and in dimension $N = 2$ with radially
symmetry. In fact, we have an explicit formulation of the
stationary states,
\begin{equation}\label{eq:equi}
\mu(x) = \dfrac1{\pi(1 + |x|^2)^N}\,, \quad N = 1,2\, .
\end{equation}
This coincides with the equality cases in the Logarithmic-HLS
inequality.

In the subcritical case $\chi <1$, solutions are known to converge
to unique self-similar profiles \cite{BDP}. For studying
convergence towards self-similarity, it is generally useful to
rescale the space and time variables in the subcritical regime
$\chi<1$. The Keller-Segel system rewrites as
\begin{equation} \label{eq:KSres}
\dfrac{\partial \rho}{\partial t}  =  \dfrac 1N \Delta\rho + 2\chi
\nabla\cdot \left[ \rho \left(\nabla \log |x| * \rho \right)
\right] + \nabla\cdot \left[ x \rho \right]\, , \quad t>0\, ,
\quad x\in \RR^N\, ,
\end{equation}
and the free energy is complemented with a quadratic confinement
potential:
\begin{equation}\label{freerescaled}
\mathcal F_{\resc}[\rho] = \mathcal F[\rho] + \dfrac12
\int_{\RR^N} |x|^2 \rho(x)\, dx\, .
\end{equation}
Due to the change of variables, self-similar solutions correspond
to equilibrium solutions of \eqref{eq:KSres}. The rate of
convergence towards equilibrium for \eqref{eq:KSres} in the
subcritical case was recently studied in
\cite{DolbeaultEscobedoetal} where the same rate as for the heat
equation was obtained for small mass.

Let us finally mention that both \eqref{eq:KS} and
\eqref{eq:KSres} are gradient flows of the free energy functionals
\eqref{eq:functional} and \eqref{freerescaled} respectively, when
the space of probability measures is endowed with the euclidean
Wasserstein metric $W_2$. We refer to the seminal papers \cite{JKO,Otto} and to \cite{AmbrosioGigliSavare02p} for a
general theory. For instance, we can write \eqref{eq:KS} in short
as
\begin{equation}
\dot\rho(t) = - \nabla_{W_2} \mathcal F[\rho(t)]\, .
\label{eq:gradient flow}
\end{equation}
This assertion was made rigorous in \cite{BCC}, where the
variational minimizing movement scheme
\cite{AmbrosioGigliSavare02p} was shown to converge for
\eqref{eq:KS}. This fact allows us to consider a way of measuring
the distance towards equilibrium or self-similarity intimately
related to the evolution due to \eqref{eq:gradient flow}. In fact,
we will show that optimal transport tools are key techniques to
describe this behavior at least in the one dimensional case
\eqref{eq:KS1D} and in the radial case in two dimensions
\eqref{eq:radialKS}.

In order to investigate further the bounds of the free energy
functional leading to the dichotomy discussed above and the
characterization of the critical profiles, the Logarithmic-HLS
inequality proved in \cite{Be,CarLos} is essential.  In Section 2,
we show an alternative proof based on optimal transport tools in
the one dimensional case and in the radial case in two dimensions.
There is another recent proof of this inequality with sharp
constants in the two dimensional case by fast diffusion flows
\cite{CCL}. The Logarithmic-HLS inequality can be restated with
our notations as:

\begin{theorem}[Logarithmic HLS inequality] \label{thm:logHLS}
Assume $N = 1,2$. The functional $\mathcal F $ is bounded from
below. The extremal functions are uniquely given by
\eqref{eq:equi} up to dilations in the set of probability
densities with zero center of mass.
\end{theorem}

In short, we demonstrate that any critical point of the free
energy is in fact a global minimizer. This is a property which
holds true for convex functionals, although the functional
$\mathcal F$ is not displacement convex in the sense of McCann \cite{McC2}.

The ideas behind the proof of the sharp Logarithmic-HLS
inequality allow us to tackle the rate of convergence in $W_2$ by
similar methods for the rescaled version \eqref{eq:KSres} in one
dimension and for radial densities in the two dimensional case
provided $\chi<1$. We prove in Section 3 the following result.

\begin{theorem}[Long-time asymptotics]\label{thm:dynamics}
Assume that $N = 1,2$ being the initial data $\rho_0$ radially
symmetric if $N=2$. In the subcritical case $\chi<1$, solutions of
\eqref{eq:KSres}
in the rescaled
variables converge exponentially fast towards the unique equilibrium
configuration $\nu$. More precisely, the following estimate holds
true
\begin{equation*}
\dfrac d{dt} W_2(\rho(t),\nu)^2 \leq -2W_2(\rho(t),\nu)^2 \, .
\end{equation*}
\end{theorem}

Surprisingly enough, the rate of convergence that we obtain does
not depend on the parameter $\chi$. Our estimate is uniform as
long as $\chi$ remains subcritical and is equal to the rate of
convergence towards self-similarity for the heat equation. This is
due to the fact that entropy and interaction contributions cancel
each other, and only the confinement contribution remains yielding
a uniform estimate. Although convergence is likely to be uniform,
notice that the asymptotic profile becomes more and more singular
as $\chi\to 1^-$, as shown by the simple second moment identity
\begin{equation}\label{eq:2nd moment nu}
\int_\RR |a|^2\nu(a)\, da = 1 - \chi\, .
\end{equation}

Finally, we devote Section 4 to propose an alternative method of
measuring the distance towards self-similarity in the one
dimensional case. We make a connection between the one dimensional
modified Keller-Segel model \eqref{eq:KS1D} and certain
Boltzmann-like equations used in granular gases and
wealth-distribution models, see \cite{CT,DMT} and the references
therein. This connection is due to the fact that \eqref{eq:KS1D}
can be written in Fourier variables like the referred Boltzmann
equations. Following the ideas of \cite{CT} we prove that equation
\eqref{eq:KS1D} is indeed a contraction for the so-called Fourier
distances defined in Section 4.

\begin{theorem}\label{thm:contraction Fourier}
Assume $\chi< 1$ and the initial data have finite second moments. The one-dimensional Keller-Segel system
\eqref{eq:KS1D} is a contraction for the distance  $d_1$.
It is a uniformly strict contraction in the rescaled frame, with a contraction factor which does not depend on~$\chi$.
\end{theorem}

\section{An alternative proof of the logarithmic HLS inequality}

\subsection{Preliminaries on Optimal Transport Tools}
Let $\mu$ and $\rho$ be two density probabilities. According to
\cite{B,McC1} there exists a convex function $\psi$ whose gradient
pushes forward the measure $\mu(a)\, da$ onto $\rho(x)\, dx$:
$\nabla \psi\# \left(\mu(a)\, da\right) = \rho(x)\, dx$. This
convex function satisfies the Monge-Ampre equation in the weak
sense,
\begin{equation*}
\mu(a) = \rho(\nabla\psi(a))\det D^2\psi(a)\, .
\label{eq:MongeAmpere}
\end{equation*}
Regularity of the transport map is a big issue in general. Here we
will use the fact that the Hessian measure $\det{}_H D^2\psi(a)$
can be decomposed in an absolute continuous part $\det{}_A
D^2\psi(a)$ and a positive singular measure \cite[Chapter 4]{Villani}. In
particular we have $\det{}_H D^2\psi(a) \geq \det{}_A D^2\psi(a)$.
The formula for the change of variables will be important when
dealing with the entropy contribution. For any measurable function $U$, bounded below such that $U(0) = 0$ we have
\cite{McC2,Villani}
\begin{equation}\label{eq:change variables}
\int_{\RR^N} U(\rho(x)) \, dx = \int_{\RR^N}
U\left(\dfrac{\mu(a)}{\det{}_A D^2\psi(a)}\right)\det{}_A D^2\psi(a)\, da\, .
\end{equation}

In fact this paper will only be concerned with the one-dimensional
case, and the two-dimensional radial case. The complexity of
Brenier's transport problem dramatically reduces in both cases. In
dimension one, the transport map $ \phi'$ is explicitely given by: $ \psi'(a) = X\circ A^{-1}(a)$ where $X$ and $A$ denote respectively the pseudo-inverse
cumulative distribution function of the densities $\rho$ and
$\mu$. The singular part of the positive measure $\psi''$
corresponds to having holes in the support of the density $\rho$.

In the two-dimensional radial case, the Brenier's map can be
expressed as the one-dimensional transport between the densities
$2\pi\mu(a)a^{d-1}\, da$ and $2\pi\rho (r)r^{d-1}\, dr$. The
determinant of the Hessian is given by
\begin{equation*}
\det{}_H D^2 \psi(a) = \dfrac 1 {2a} \dfrac d{da} \left( \psi'
\right)^2(a)\, ,
\end{equation*}
where the derivative of $(\psi')^2$ has to be understood in the distributional sense.

The following Lemma will be used to estimate the interaction
contribution in the free energy, and in the evolution of the
Wasserstein distance. For notational convenience we denote the
convex combination of $a$ and $b$ by $[a,b]_t = (1-t)a + t b$.

\begin{lemma}\label{lem:interaction}
Let $K:(0,\infty)\to \RR$ be an increasing and concave function
such that $\lim_{z\to 0} F(z) = -\infty$. Then
\begin{equation} \label{eq:jensen 1D}
K\left( \int_{0}^1 \psi''([a,b]_t)\, dt  \right) \geq  \int_{0}^1K\left( \psi_{\AC}''([a,b]_t)  \right)\, dt\, .
\end{equation}
Equality is achieved in \eqref{eq:jensen 1D} if and only if the distributional derivative of the
transport map $\psi''$ is a constant function.

Analogously in the two-dimensional radially symmetric case we deduce
\begin{equation} \label{eq:jensen radial}
K\left( \int_{0}^1 \det{}_H D^2\psi([a,b]_t)\, dt  \right)
\geq  \int_{0}^1K\left( \det{}_A D^2\psi([a,b]_t)  \right)\,
dt\, .
\end{equation}
Equality is achieved in \eqref{eq:jensen radial} if and only if $\psi'$ is a multiple of the identity.
\end{lemma}

\begin{proof}
We have on the one hand $ \psi''\geq \psi_\AC ''$. We next use the
concavity of $K$ to conclude. Equality occurs if $\psi''$ is
absolutely continuous and if $\psi_\AC''$ is constant. In the
two-dimensional case we use $\det{}_H D^2\psi(a) \geq \det{}_A
D^2\psi(a)$.
\end{proof}

Optimal transport is a powerful tool for reducing functional
inequalities onto pointwise inequalities ({\em e.g.} matrix
inequalities). We highlight for example the seminal paper by
McCann \cite{McC2} where the displacement convexity issue for some
energy functional is reduced to the concavity of $\det^{1/N}$. We
also refer to the works of Barthe \cite{Barthe1,Barthe2} and
Cordero-Erausquin {\em et al.} \cite{CoNaVi}. We require
simple pointwise inequalities which are extensions of the
classical Jensen's inequality.

\begin{lemma}\label{lem:jensen quadratic}
 We have the following convex-like inequality for some exponent $\gamma>0$ and any positive  $u,v,\alpha,\beta$,
\begin{equation} \label{eq:micro lemma resc}
 \alpha  \left( \dfrac{u+v}{2} \right)^{-\gamma}   - {\beta }  \left(\dfrac{u+v}{2}\right)^\gamma    \leq  (\alpha+\beta) \left(\dfrac{    u^{-\gamma} + v^{-\gamma} }2 \right)  - 2\beta \, .
\end{equation}
Equality occurs if and only if $u=v=1$. The continuous version
reads as follows. For any measurable function $u:(0,1)\to
(0,+\infty)$:
\begin{equation}
\alpha   \left(  \int_{0}^1 u(t)\, dt  \right)^{-\gamma}   -
\beta  \left(\int_{0}^1 u(t)\, dt\right)^\gamma     \leq  \left(\alpha
+ \beta\right)   \int_{0}^1 ( u(t) )^{-\gamma}\, dt - 2\beta
\, , \label{eq:variant Jensen log diff}
\end{equation}
\end{lemma}

\begin{proof}
We only prove \eqref{eq:micro lemma resc}. The continuous version
\eqref{eq:variant Jensen log diff} is obtained by an approximation
procedure. We introduce the auxiliary function $J$ defined as
follows.
\[
J(u,v) = (\alpha+\beta) \left(\dfrac{    u^{-\gamma} + v^{-\gamma} }2 \right) - \alpha  \left( \dfrac{u+v}{2} \right)^{-\gamma} + {\beta }  \left(\dfrac{u+v}{2}\right)^\gamma\, .
\]
Clearly, $J$ diverges towards $+\infty$ as $u\to 0$ or $v\to 0$,
and as $u\to \infty$ or $v\to \infty$, and $J$ is bounded below.
Then there exists at least one critical point. Any critical point
$(u_0,v_0)$ satisfies
\[
\left\{\begin{array}{r} -\gamma \dfrac{\alpha + \beta}2 u_0^{-\gamma - 1} + \gamma \dfrac\alpha2 \left( \dfrac{u_0+v_0}{2} \right)^{-\gamma - 1}  + \gamma \dfrac\beta 2 \left( \dfrac{u_0+v_0}{2} \right)^{\gamma - 1} = 0 \, ,  \\ -\gamma \dfrac{\alpha + \beta}2 v_0^{-\gamma - 1} + \gamma \dfrac\alpha2 \left( \dfrac{u_0+v_0}{2} \right)^{-\gamma - 1}  + \gamma \dfrac\beta 2 \left( \dfrac{u_0+v_0}{2} \right)^{\gamma - 1} = 0 \, .
\end{array}\right.
\]
Hence $u_0 = v_0$ and
\[ - \dfrac{\alpha + \beta}2 u_0^{-\gamma - 1} +   \dfrac\alpha2 u_0^{-\gamma - 1}  +  \dfrac\beta 2 u_0^{\gamma - 1} = 0\, . \]
We conclude that $u_0^{-\gamma} = u_0^{\gamma}$. Therefore the unique critical point of $J$ is $ (1,1)$.
\end{proof}

\subsection{The one-dimensional case}

The novelty here is contained in the proof of the logarithmic HLS
inequality. This brings no information by itself since the
uniqueness of the extremal functions is already known
\cite{CarLos}. We show below that the logarithmic HLS inequality
is a simple consequence of the Jensen's inequality. However our
proof relies on the existence of a critical point of the free
energy $\mathcal F$. In short, we demonstrate that any critical
point of the free energy is in fact a global minimizer. This is a
property which holds true for convex functionals. However the
functional here is not convex.

Our first Lemma is a reformulation of the Euler-Lagrange equation
for the extremal function \eqref{eq:equality1D}.
\begin{lemma}[Characterization of extremal functions]
The critical profiles
satisfy the following identity,
\begin{equation} \label{eq:charac. log}
\mu(p) = \int_{\RR}\int_{0}^1 \mu(p -tq)\mu(p-tq+q)\,
dtdq\, .
\end{equation}
In the subcritical regime $\chi<1$, the equilibrium in the
rescaled frame satisfies the following identity,
\begin{equation}
\label{eq:charac. log quadratic} \nu(p) = \int_{q\in
\RR}\int_{0}^1 \left( \chi  + \dfrac{|q|^2}2 \right)
\nu(p-tq)\nu (p-tq+q)\, dtdq\, .
\end{equation}
\end{lemma}

\begin{proof}
The formulation \eqref{eq:charac. log} is equivalent to
integrating once the equation for  the critical profile. We
integrate equation \eqref{eq:equality1D} against some test
function $\varphi$.
\begin{align*}
 \int_\RR \varphi'(p) \mu(p)\, dp  & =  2 \iint_{\RR\times\RR} \dfrac{\varphi(x)}{x-y} \mu(x)\mu(y)\, dxdy \\  & =   \iint_{\RR\times\RR} \dfrac{\varphi(x)-\varphi(y)}{x-y} \mu(x)\mu(y)\, dxdy\\
&=  \iint_{\RR\times\RR}\int_{0}^1 \varphi'\left([x,y]_t\right)
\mu(x)\mu(y)\, dtdxdy  \\ & =   \int_{
\RR}\varphi'(p)\left\{\int_{\RR}\int_{0}^1  \mu(p - tq
)\mu(p-tq+q)\, dtdq\right\}\, dp\, ,
\end{align*}
where we have finally used the change of variables: $(x,y) \mapsto
(p = [x,y]_t , q = y-x) $. This holds true for any derivative
$\varphi'$, so we obtain identity \eqref{eq:charac. log} up to a constant. Since both sides of \eqref{eq:charac. log} have mass one, the constant is zero.
The identity \eqref{eq:charac. log quadratic} is obtained in a
similar way.
\end{proof}

\begin{proof}[Proof of Theorem \ref{thm:logHLS}.] Applying the change of variables formula \eqref{eq:change variables} for $x = \psi'(p)$, the functional $\mathcal F$ rewrites as follows,
\begin{align*}
\mathcal F[\rho] - \mathcal F[\mu] & = \!\!\int_\RR\!\!
\log\left(\dfrac{\mu(a)}{\psi_{\AC}''(a)}\right) \mu(a)\, da
\!+\!\! \iint_{\RR\times\RR}\!\!\!\!\!\! \log|\psi'(a)-\psi'(b)|
\mu(a)\mu(b)\, da db - \mathcal F[\mu]
\\&= - \int_\RR \!\!\log\left( {\psi_{\AC}''(a)}\right) \mu(a)\, da +
\iint_{\RR\times\RR}\!\!\!\!\!\!
\log\left(\dfrac{\psi'(a)-\psi'(b)}{a-b}\right) \mu(a)\mu(b)\, da
db
\\&= - \int_\RR \!\!\log\left( {\psi_{\AC}''(a)}\right) \mu(a)\, da + \iint_{\RR\times\RR}\!\!\!\!\!\! \log\left(\int_{0}^1\psi''([a,b]_t)\, dt\right) \mu(a)\mu(b)\, da db
\end{align*}
Using Lemma \ref{lem:interaction} for $K = \log z$ which is
increasing and concave, we deduce
\begin{align*}
\mathcal F[\rho] - \mathcal F[\mu]  \geq & - \int_\RR
\!\!\log\left( {\psi_{\AC}''(p)}\right) \mu(p)\, dp +
\iint_{\RR\times\RR}\int_{0}^1 \!\!
\log\left(\psi_{\AC}''([a,b]_t)\right) \mu(a)\mu(b)\, dt da db
\\ =& - \int_\RR \!\!\log\left( {\psi_{\AC}''(p)}\right) \mu(p)\, dp \\
&   + \int_{\RR} \!\!\log\left(\psi_{\AC}''(p)\right) \left\{
\int_{q \in\RR}\int_{0}^1  \mu(p-tq)\mu(p-tq+q)\, dt dq\right\}\,
dp = 0 \, .
\end{align*}
Equality arises if and only if the transport map $\psi''$ is a
constant function. Such a map corresponds exactly to the dilations
of the critical profile $\mu$.
\end{proof}

It is possible to extend Theorem \ref{thm:logHLS} to the rescaled energy $\mF_\resc$ \eqref{freerescaled}.

\begin{theorem}[Logarithmic HLS inequality with a quadratic confinement] \label{thm:logHLS confinement}
Assume $N = 1,2$ and $\chi<1$. The functional $\mathcal F_\resc $ is bounded from
below. The extremal functions are unique  in the set of probability
densities with zero center of mass.
\end{theorem}

We give below the main lines of the proof following a direct argument analogous to the proof of Theorem \ref{thm:logHLS}. Note that the uniqueness of the extremal functions in dimension $N=1$ or in dimension $N=2$ in the class of radially symmetric densities is a consequence of Theorem \ref{thm:dynamics}.

\begin{proof}[Sketch of proof of Theorem \ref{thm:logHLS confinement}]
The key point consists in replacing the Jensen's inequality  with the
following convex-like inequality. For any positive
$u,v,\alpha,\beta$ the following inequality holds true.
\begin{equation} \label{eq:Jensen quadratic}
\alpha \log \left( \dfrac{u+v}{2} \right)   + {\beta }
\left(\dfrac{u+v}{2}\right)^2   \geq  (\alpha+2\beta)
\left(\dfrac{  \log  u +\log v }2 \right) + \beta\, ,
\end{equation}
Equality occurs if and only if $u = v= 1$. It reduces to the usual
Jensen's inequality when $\beta = 0$. The proof of \eqref{eq:Jensen quadratic} is analogous to Lemma \ref{lem:jensen quadratic}. The proof of uniqueness for the extremal functions of $\mF_\resc$ is a mixture between the proofs of Theorem \ref{thm:logHLS} and Theorem \ref{thm:dynamics}.
\end{proof}

\subsection{The two-dimensional case}

We restrict to radially symmetric functions in the two-dimensional
case due to decreasing rearrangement \cite{liebloss,Be,CarLos}. We recall the Newton's theorem for Poisson potential: the
field induced by a radially symmetric distribution of masses
outside a given ball is equivalent to the field induced by a point
at the center of the ball \cite{liebloss}. Equivalently it reads
\begin{equation}\label{eq:Newton}
\dfrac12 \int_{\theta = 0}^{2\pi} \log\left(r^2 + s^2 -
2rs\cos(\theta)\right) \, d\theta = 2\pi \log  \max(r,s)  \, .
\end{equation}
As a consequence we can rewrite the functional $\mF$ simpler
under radial symmetry:
\begin{equation*}
\label{eq:log func radial} \dfrac1{2\pi}\mF[\rho] =
\dfrac12\int_{\RR_+} \rho(r)\log(\rho(r))\, r dr +  \chi
\int_{\RR_+} \rho(r) M[\rho](r) \log(r)Ê \,  rdr\, .
\end{equation*}

The following characterizations are direct
consequences of \eqref{eq:equality2D} and \eqref{eq:radialKS}.
\begin{lemma}[Characterization of extremal functions under radial symmetry]
The critical profiles satisfy the following identity
\begin{equation} \label{eq:StStcharacterization radial}
 \dfrac12  \mu(b)  =  2 \int_{b}^{+\infty} \mu(a) ÊM[\mu](a)   \dfrac{1}a \, da  \, .
\end{equation}
In the subcritical regime $\chi<1$, the radially-symmetric equilibrium satisfies the following identity
\begin{equation}
\dfrac12  \nu(b)  = \int_{b}^{+\infty} \nu(a)\left( 2\chi
M[\nu](a)\dfrac{1}a + a \right) \, da\, .
\label{eq:StStcharacterization radial resc}
\end{equation}
\end{lemma}

We are now ready to examinate the logarithmic
Hardy-Littlewood-Sobolev inequality in the two-dimensional radial
setting.
\begin{proof}[Proof of Theorem \ref{thm:logHLS}.]
We apply the change of variables formula \eqref{eq:change variables} for $ r = \psi'(a)$ to get:
\begin{equation*}
\dfrac1{2\pi} \mF[\rho] = \dfrac12 \int_{\RR_+} \mu(a )
\log\left(\dfrac{\mu(a)}{\det{}_A D^2\psi(a)}\right)     a\, da +
2  \int_{\RR_+}  \mu(a)    M[\mu](a)   \log\left(  \psi'(a)
\right)   \, a  da \, ,
 \end{equation*}
 where we have used $M[\rho](r) = M[\mu](a)$. We have consequently,
 \begin{align}
\dfrac1{2\pi}  \mF[\rho] - \dfrac1{2\pi}  \mF[\mu] = & - \dfrac12 \int_{\RR_+} \mu(a ) \log\left(\det{}_A D^2\psi(a)\right)   a\, da  \nonumber \\
& +  \int_{\RR_+}  \mu(a)   M[\mu](a)    \log\left( \dfrac{\left( \psi'\right)^2(a)}{a^2} \right)    a\,   da\, . \label{eq:2D radial core}
\end{align}
The last contribution of \eqref{eq:2D radial core} can be
evaluated using Lemma \ref{lem:interaction}
\begin{align*}
& \int_{\RR_+}   \mu(a)    M[\mu](a)   \log\left(  \int_{0}^a \left( \det{}_H D^2 \psi(b) \right) \dfrac  {2b}{a^2} \, db \right)   a \,   da \\
   \geq &  \int_{\RR_+}\int_{0}^a   \mu(a)    M[\mu](a)   \log\left( \det{}_A D^2 \psi(b) \right)  \dfrac{2b}a \,   db da \\
  = & \int_{\RR_+}  \log\left(\det{}_A D^2 \psi(b)  \right)
\left\{ 2 \int_{b}^{+\infty}    \mu(a)    M[\mu](a) \dfrac1a\, da\right\}
b\, db \,.
\end{align*}
We obtain from the characterization
\eqref{eq:StStcharacterization radial} $ \mF[\rho] \geq \mF[\mu]$. Again equality occurs if and only if the transport map $\psi'$ is
a multiple of the identity.
\end{proof}

\subsection{Obstruction in dimension higher than three}

We explain in this Section why the above strategy fails to work in
dimension higher than three, even in the radially-symmetric
setting. A first remark is that Newton's Theorem is not valid,
since the logarithm kernel is not the fundamental solution of the
Poisson equation, although this is not essential as shown in dimension one. It turns out that our strategy works
fine for any interaction kernel $W(x) = |x|^k/k$, for $k\in
(-N,2-N]$. The case $k = 0$ corresponds to $W = \log|x|$. The case
$k = -N$ is critical for integrability reasons. The case $k = 2-N$
is exactly the harmonic case for which the Newton's Theorem holds
true. We refer to \cite{Calvez.Carrillo.I} for details in the case
$k\in (-N,2-N]$. Hence the case $k = 0$ is out of range when
$N\geq 3$. We sketch below where some obstruction appears when $N
= 3$.

The identity which generalizes \eqref{eq:Newton} reads as
follows. If $r>s$ we have
\begin{align*}
\dfrac12 \int_{\theta = 0}^{\pi} \log\left(r^2 + s^2 -
2rs\cos(\theta)\right) \sin(\theta)\, d\theta
= \dfrac12 H\left(\dfrac{s}{r}\right) + 2 \log r\, ,
\end{align*}
where $H$ is defined as follows for $t\in (0,1)$,
\[H(t) = \dfrac12  \dfrac{(1+t)^2}{t} \log\left(1+t\right) - \dfrac12  \dfrac{(1-t)^2}{t} \log\left(1-t\right) - 2\, . \]
To continue our strategy, it is required to decouple the variables
$r$ and $s$, and more precisely to make the quantity $r^N - s^N$
appearing. As a matter of fact this is homogeneous to the
determinant (under radial symmetry), which is the key quantity to
look at in dimension higher than two. Therefore we seek a
convex-like inequality
\begin{equation*}
\label{eq:convex-like}
H(t) \geq \alpha + \beta \log\left(1 - t^N\right)\, , \end{equation*}
where $\alpha$ and $\beta$ are suitable constants determined by
zero and first order conditions. If we denote $\varphi(t) = \log(1
- t^N)$, this is equivalent to say that $H\circ \varphi^{-1}$ is
convex. However simple computations show that it is indeed a concave function. In the case of an interaction
kernel having homogeneity $k\in (-N,2-N]$ we show in
\cite{Calvez.Carrillo.I} that the corresponding function
$H\circ\varphi^{-1}$ is convex.

\section{Exponential convergence towards the self-similar profile}

\subsection{The one-dimensional case}
To illustrate the strategy of proof of Theorem \ref{thm:dynamics}, we show a formal computation in the critical case $\chi = 1$.
Up to our knowledge, the regularity of solutions
under very weak assumptions is still an open problem
In particular it
is not known whether the solutions satisfy the identity
\eqref{eq:decreasing energy} or not. So the following computation
is questionable because the velocity field $\partial_x( \log
\rho(t,x) + 2  \log |x|*\rho(t,x) ) $ is not clearly defined in
$L^2(\rho(t,x)dx)$.

We compute formally the evolution of the Wasserstein distance to
one of the equilibria \eqref{eq:equi} in the critical case $\chi =
1$. Notice that equilibria are infinitely far from each other with
respect to the Wasserstein distance \cite{BCC-Critical}. Using the
gradient flow structure with respect to $W_2$, one obtains the
following formula for the derivative of $F(t) =
W_2(\rho(t),\mu)^2$, see \cite[Chapter 8]{Villani} and
\cite{AmbrosioGigliSavare02p}.
\begin{align*}
\dfrac12 \dfrac d{dt} F(t) & = \int_\RR (\phi'(t,x) - x)\left( \partial_x \left( \log \rho(t,x) + 2  \log |x|*\rho(t,x) \right) \right) \rho(t,x)\, dx \\
& = - \int_\RR \phi''(t,x) \rho(t,x)\, dx +   \iint_{\RR\times\RR} \dfrac{\phi'(t,x) - \phi'(t,y)}{x-y}\rho(t,x)\rho(t,y)\, dxdy \\
& = - \!\int_\RR \left(\psi''(t,a)\right)^{-1} \!\!\mu(a)\, da + \! \iint_{\RR\times\RR} \left(\dfrac{\psi'(t,a) - \psi'(t,b)}{a-b}\right)^{-1}\!\!\mu(a)\mu(b)\, dadb \\
& \leq  - \int_\RR \left(\psi''(t,a)\right)^{-1} \!\!\mu(a)\, da +
\iint_{\RR\times\RR}   \int_{0}^1
\left(\psi''(t,[a,b]_s)\right)^{-1} \!\!\mu(a) \mu(b)\, ds da db\,
.
\end{align*}
We recognize the characterization \eqref{eq:charac. log}. Hence,
we have at least formally $F'(t) \leq 0$. Observe that the Lemma
\ref{lem:interaction} has been used with $K(z) = -z^{-1}$.

The same strategy is valid in the subcritical case $\chi<1$ for
which we know that solutions are regular enough to ensure the
validity of the computations. As a matter of fact, the density
$\rho(t,x)$ is everywhere positive and thus $\psi''$ is
absolutely continuous. On the other hand the dissipation of energy
is well-defined and the dissipation estimate \eqref{eq:decreasing
energy} holds true \cite{BDP}.

\begin{proof}[Proof of Theorem \ref{thm:dynamics}]
We compute the evolution of $F(t) = W_2(\rho(t),\nu)^2$:
\begin{align*}
\dfrac12 \dfrac d{dt}F(t)  = & \int_\RR (\phi'(t,x) - x)\left( \partial_x \left( \log \rho(t,x) + 2\chi \log |x|*\rho(t,x) + \dfrac{|x|^2}2 \right) \right) \rho(t,x)\, dx \\
 =& - \int_\RR \phi''(t,x) \rho(t,x)\, dx + \chi \iint_{\RR\times\RR} \dfrac{\phi'(t,x) - \phi'(t,y)}{x-y}\rho(t,x)\rho(t,y)\, dxdy \\
&   - \dfrac12 \iint_{\RR\times\RR} (\phi'(t,x) - \phi'(t,y))(x-y)\rho(t,x)\rho(t,y)\, dxdy \\
& + 2 \int_\RR \phi'(t,x) x \rho(t,x)\, dx  + 1 - \chi - \int_\RR |x|^2\rho(t,x)\, dx   \, ,
\end{align*}
where we have used the fact that the center of mass is zero to
double the variables. We rewrite each contribution using the
reverse transport map $\psi'$:
\begin{align*}
\dfrac12 \dfrac d{dt}F(t)  = & - \int_\RR \left(\psi''(t,a)\right)^{-1}\!\! \nu(a)\, da + \chi \iint_{\RR\times\RR} \left(\dfrac{\psi'(t,a) - \psi'(t,b)}{a-b}\right)^{-1}\!\!\!\!\!\nu(a)\nu(b)\, dadb \\
&  - \dfrac12 \iint_{\RR\times\RR} |a-b|^2\dfrac{ \psi'(t,a) - \psi'(t,b) }{a-b}\nu(a)\nu(b)\, da db\\
&   + 1-\chi - \int_\RR |\psi'(t,a)|^2\nu(a)\, dx + 2
\int_\RR a \psi'(t,a) \nu(t,a)\, da
\\ \leq & \iint_{\RR\times\RR}  \left[ \left( \chi + \dfrac{|a-b|^2}2 \right) \int_{0}^1 \left(\psi''(t,[a,b]_s)\right)^{-1} ds  - |a-b|^2 \right] \nu(a) \nu(b)\, da db\\
&  - \int_\RR \left(\psi''(t,a)\right)^{-1} \nu(a)\, da  + 2
\int_\RR |a|^2 \nu(a)\, da - \int_\RR |\psi'(t,a) - a|^2\nu(a)\,
da \, ,
\end{align*}
where we have used that the second moment of the stationary state
is explicitly given by \eqref{eq:2nd moment nu}. Applying now
\eqref{eq:variant Jensen log diff} for $\gamma = 1$, $\alpha =
\chi $ and $\beta = |a-b|^2/2$, we deduce
$$
\dfrac12 \dfrac d{dt}F(t) \leq - \int_\RR |\psi'(t,a) -
a|^2\nu(a)\, da = - W_2(\rho(t),\nu)^2 = - F(t)\, ,
$$
giving the desired inequality.
\end{proof}

\subsection{The two-dimensional radially-symmetric case}

Proving convergence towards a self-similar profile in the rescaled
logarithmic case under radial symmetry goes as previously.

\begin{proof}[Proof of Theorem \ref{thm:dynamics}.]
The virial computation reads equivalently
\begin{align*}
\dfrac12 \dfrac d{dt} \int_{\RR_+} \rho(t,r) r^3 \, dr = &  -
\int_{\RR_+} r \left( \dfrac12\partial_r \log\rho(t,r) + 2\chi
{M[\rho](t,r)}\dfrac1{r} + r\right) \rho(t,r) r\, dr
\\ = & \int_{\RR_+} \rho(t,r) r\, dr - 2\chi \int_{\RR_+}
M[\rho](t,r) \rho(t,r) r \, dr -\! \int_{\RR_+} \!\! \rho(t,r)
r^3\, dr
\\
= & 1 - \chi - \int_{\RR_+}  \rho(t,r) r^3\, dr = \int_{\RR_+}
\nu(a) a^3 \, da - \int_{\RR_+}  \rho(t,r) r^3\, dr
\end{align*}
We compute again the evolution of the Wasserstein distance $F(t) =
W_2(\rho(t),\nu)^2$.
\begin{align*}
& \dfrac12\dfrac{d}{dt} F(t)
=  \int_{\RR_+} \left(\phi'(r) - r\right)\left( \dfrac12\partial_r \log\rho(t,r) + 2\chi {M[\rho](t,r)}\dfrac1{r} + r\right) \rho(t,r) r\, dr\\
 = & \dfrac12\int_{\RR_+} r \phi'(r) \partial_r \rho(t,r)\, dr +2 \chi\int_{\RR_+} \phi'(r) M[\rho](t,r)\rho(t,r)\, dr -  \int_{\RR_+} \phi'(r) \rho(t,r)r^2 \, dr \\
& + \int_{\RR_+} \nu(a) a^3 \, da - \int_{\RR_+}  \rho(t,r) r^3\,
dr  + 2 \int_{\RR_+} \phi'(r) \rho(t,r) r^2 \, dr
\\
 \leq &  2\chi \int_{\RR_+} \left(\int_0^a\det  D^2\psi(b)\dfrac{2b}{a^2}\, db\right)^{-1/2} M[\nu](a)\nu(a) a \, da \\
& - \int_{\RR_+} \left(\dfrac1{\det D^2\psi(b)} \right)^{1/2} \nu(b) b \, db  - \int_{\RR_+} \left(\int_0^a\det  D^2\psi(b)\dfrac{2b}{a^2}\, db\right)^{1/2} a^2\nu(a) a da \\
&  +2 \int_{\RR_+} \nu(a) a^3 \, da -  \int_{\RR_+} |\phi'(r) - r
|^2 \rho(t,r) r\, dr \, .
\end{align*}
The last step in the inequality is a consequence of the arithmetic
and geometric means inequality: $-\partial_r(r\phi'(r))/r = -
\phi'(r)/r - \phi''(r) \leq - 2 (\phi''(r)\phi'(r)/r)^{1/2}$. Next
we use Lemma \ref{lem:jensen quadratic} to handle the interaction
contribution. More precisely, we choose $\gamma = 1/2$, $\alpha =
2\chi M[\nu](a)$ and $\beta = a^2$. One gets finally:
\begin{align*}
\dfrac12\dfrac{d}{dt} F(t)  \leq &  - F(t) - \int_{\RR_+} \left(\det  D^2\psi(b)\right)^{-1/2} \nu(b) b \, db \\
&  + \int_{\RR_+}  \int_{b}^{+\infty}  \left( 2\chi M[\nu](a) +
a^2  \right)\left(\det  D^2\psi(b)\right)^{-1/2}\dfrac{2b}{a^2}
\nu(a) a \, db da \, .
\end{align*}
We conlude using characterization \eqref{eq:StStcharacterization
radial resc}.
\end{proof}


\section{Contraction in the one-dimensional case}

The aim of this Section is to point out the peculiar structure of
the modified one-dimensional Keller-Segel system \eqref{eq:KS1D}.

\begin{lemma}
Equation \eqref{eq:KS1D} rewrites in Fourier variables as:
\begin{equation}\label{eq:1DKS Fourier}
\partial_t \hat \rho(t,\xi) = |\xi|^2 \left( - \hat \rho(t,\xi) + \chi \int_{0}^1 \hat\rho(t,\sigma\xi) \hat \rho(t,(1-\sigma)\xi)\, d\sigma\right)\, .
\end{equation}
\end{lemma}

\begin{proof}
We test equation \eqref{eq:KS1D} against $\exp\left( i \xi x
\right)$:
\begin{align*}
\dfrac \partial{\partial t}\widehat{\rho}(t,\xi) & =
\int_\RR \left( \partial_{xx} \rho(t,x) + 2\chi \partial_x\left(\rho(t,x)\left( \pv\dfrac1x\right) * \rho(t,x)\right)  \right)e^{i\xi x}\, dx \\
& = - |\xi|^2 \hat \rho(t,\xi) - \chi i \xi \iint_{\RR\times \RR} \rho(t,x) \dfrac{e^{i\xi x} - e^{i\xi y}}{x-y}\rho(t,y)\, dx dy\\
& = - |\xi|^2 \hat \rho(t,\xi) + \chi |\xi|^2 \iint_{\RR\times \RR} \rho(t,x) \left(\int_{0}^1 e^{i\xi[x,y]_\sigma}\, d\sigma \right)\rho(t,y)\, dx dy \\
& = - |\xi|^2 \hat \rho(t,\xi) + \chi |\xi|^2 \int_{0}^1
\!\!\!\left(\int_{\RR} \rho(t,x) e^{i(1-\sigma)\xi x }\, dx
\right)\left(\int_{\RR} \rho(t,y) e^{i\sigma\xi y }\, dy \right)
\, d\sigma \, ,
\end{align*}
which gives the desired formulation.
\end{proof}

According to \eqref{eq:1DKS Fourier} the information propagates
from lower to higher frequencies. The evolution of
$\hat\rho(t,\xi)$ requires the knowledge of lower frequencies
$|\xi'|<|\xi|$ due to the integral contribution. This is of
particular importance for designing a numerical scheme. Indeed
there is no loss of information after truncation of the frequency
box.

\begin{figure}
\begin{center}
\includegraphics[width = .8\linewidth]{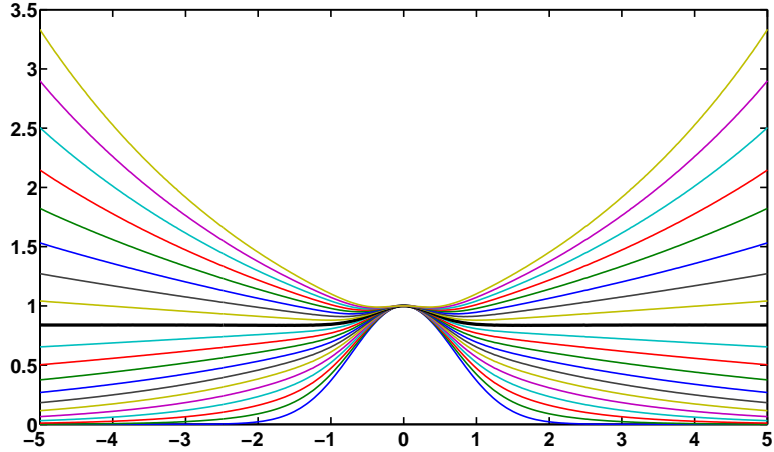}
\caption{Simulation of \eqref{eq:1DKS Fourier} in the
supercritical case $\chi>1$ for successive times. The blow-up time
is plotted in bolded dark. } \label{fig:contraction}
\end{center}
\end{figure}

\begin{remark}[Analogy with 1D Boltzmann]
It is worthy to metion that the integral operator in the
right-hand-side of \eqref{eq:1DKS Fourier} is reminiscent of the
homogeneous Boltzmann equations in 1D used for granular gases
\cite{CT} or wealth distribution models \cite{DMT} in Fourier
variables.
\end{remark}

\begin{remark}[Evidence for blow-up in the supercritical case]
We can directly notice the occurence of blow-up when $\chi>1$ from
\eqref{eq:1DKS Fourier}. Observe that for $|\xi|\ll1$, the
right-hand-side is equivalent to:
\begin{equation}\label{eq:wrong intuition}
\partial_t \hat \rho(t,\xi) \sim  |\xi|^2 \left( - \hat \rho(t,0) + \chi \hat \rho(t,0)^2\right) = |\xi|^2(-1+\chi)\, . \end{equation}
This argues in favor of blow-up at low modes although
misleadingly. We have plotted in Figure \ref{fig:contraction}
numerical simulation of \eqref{eq:1DKS Fourier} in the
supercritical case. Observe that blow-up arises for $|\xi|\gg1$,
on the contrary to the  misleading heuristics \eqref{eq:wrong
intuition}. The integro-differential equation \eqref{eq:1DKS
Fourier} makes perfect sense even in the supercritical regime
$\chi>1$ after the first blow-up event. However the outcoming
fonction $\hat \rho(t,\xi)$ is no longer the Fourier transform of
a probability measure. In fact the blow-up time coincides with the
formation of the first dirac mass, namely when the frequency
distribution $\widehat{\rho}(t,\xi)$ is flat at infinity. This
contradictory intuition is similar to the proof of blow-up based
on the virial identity: the second momentum provides information
at infinity but is used to prove blow-up which is a local
behaviour.
\end{remark}

Recall the definition of Fourier distances \cite{CT} as they have
been introduced for the analysis of the Boltzmann equation.
\begin{definition}[Fourier distances]
Let $\rho_1$, $\rho_2$ being two probability
measures having the same center of mass. The
$d_{1}-$distance is defined as follows:
\begin{equation}\label{def:d1}
d_{1}\left(\rho_1,\rho_2\right)  = \sup_{\xi\neq 0} \left\{
|\xi|^{-1}\left|\hat \rho_1(\xi) -
\hat\rho_2(\xi)\right|\right\} \, .
\end{equation}
\end{definition}

\begin{proof}[Proof of Theorem \ref{thm:contraction Fourier}.]
First notice that supremum in \eqref{def:d1} is attained in
$\RR\setminus\{0\}$. Clearly we have $\left|\hat \rho_1(\xi) -
\hat\rho_2(\xi)\right|\leq 2$ and $$ \hat \rho_1(\xi) -
\hat\rho_2(\xi) \sim \left(\int_\RR |x|^2[\rho_1(x)-\rho_2(x)]\,
dx \right)|\xi|^2/2 \mbox{ as } \xi\to 0.
$$

We denote $F(t) = d_1(\rho_1(t),\rho_2(t))$ and $h(t,\xi) = |\xi|^{-1}(\hat \rho_1(t,\xi) -
\hat\rho_2(t,\xi))$.
We multiply the difference between the two equations
\eqref{eq:1DKS Fourier} by $\sign(h(t,\xi))$,
\begin{equation*}
\partial_t  \left|h(t,\xi)\right|  =
|\xi|^2 \left( -  \left|h(t,\xi)\right|   + \chi \sign\left(\hat \rho_1(t,\xi) -
\hat\rho_2(t,\xi)\right)   A(t,\xi) \right)\, , 
\end{equation*}
where
\[ A(t,\xi) = |\xi|^{-1}\int_{0}^1 \hat\rho_1(t,\sigma\xi) \hat \rho_1(t,(1-\sigma)\xi)\, d\sigma - |\xi|^{-1}\int_{0}^1 \hat\rho_2(t,\sigma\xi) \hat \rho_2(t,(1-\sigma)\xi)\, d\sigma \, . \]
The self-attraction contributions are handled as follows \cite[Th.
6.3]{CT}:
\begin{align*}
\left| A(t,\xi)\right|   \leq & \,|\xi|^{-1} \int_{0}^1 \left|
\hat\rho_1(t,\sigma\xi) - \hat\rho_2(t,\sigma\xi)
\right||\hat\rho_1(t,(1-\sigma)\xi)|\,
d\sigma \\
& + |\xi|^{-1}\int_{0}^1 \left| \hat\rho_1(t,(1-\sigma)\xi) - \hat\rho_2(t,(1-\sigma)\xi)\right||\hat\rho_2(t,\sigma\xi)|\, d\sigma \\
\leq & \, d_{1 }\left(\rho_1(t),\rho_2(t)\right)
\int_{0}^1\left(\sigma  + (1-\sigma)\right) \, d\sigma = F(t)\, .
\end{align*}
We obtain finally
\[
\partial_t  \left|h(t,\xi)\right|  \leq
|\xi|^2 \left( - \left|h(t,\xi)\right| + \chi F(t) \right)\, .
\]
We deduce
\begin{align*}
|h(t+\epsilon,\xi)| \leq & \,e^{-\epsilon|\xi|^2} |h(t,\xi)| + \chi \left( 1 - e^{-\epsilon|\xi|^2} \right) \sup_{s\in (0,\epsilon)} F(t+s)\, ,  \nonumber \\
|h(t+\epsilon,\xi)| - F(t) \leq& \,\left( 1 - e^{-\epsilon|\xi|^2} \right) \left( - F(t) + \chi \sup_{s\in (0,\epsilon)} F(t+s) \right)\, ,  \nonumber \\
\limsup_{\epsilon\to0^+} \dfrac{F(t+\epsilon) - F(t)}{\epsilon}
\leq& \,(\chi - 1)
\left(\liminf_{\epsilon\to0^+}|\xi^*(t+\epsilon)|^2\right) F(t)\,
, \label{eq:Fourier contraction}
\end{align*}
where $|\xi^*(t)|$ denotes the lowest frequency moduli for which
the supremum is attained in $F(t) = \sup |h(t,\xi)|$. We have used
the continuity of $F$ to pass to the limit. Therefore we get a
contraction estimate as soon as $\chi<1$. There is no explicit
rate since we do not know how to control $|\xi^*(t)|$ from below.

We also obtain a uniform strict
contractivity in self-similar variables. The Keller-Segel equation \eqref{eq:KSres} writes
as follows in Fourier variables:
\begin{equation*}\label{eq:1DKS Fourier resc}
\partial_t \hat \rho(t,\xi) = |\xi|^2 \left( - \hat \rho(t,\xi) + \chi \int_{0}^1 \hat\rho(t,\sigma\xi) \hat \rho(t,(1-\sigma)\xi)\, d\sigma\right) - \xi \partial_\xi \hat \rho(t,\xi) \, .
\end{equation*}
We proceed as above to get:
\begin{align*}
\partial_t \left|h(t,\xi)\right|   = &
|\xi|^2 \left( - \left|h(t,\xi)\right|  +
\chi \sign\left(\hat \rho_1(t,\xi) - \hat\rho_2(t,\xi)\right)   A(t,\xi) \right) \\
& - \xi \partial_\xi \left(  \left|h(t,\xi)\right| \right) -   \left|h(t,\xi)\right|  \, .
\end{align*}
We integrate along characteristics and argue as previously,
\begin{align*}
|h(t+\epsilon,\xi)| - F(t) \leq & F(t)\left( \exp\left( - \epsilon +  \dfrac{e^{-2\epsilon} - 1}2  |\xi|^2\right) - 1 \right)\\
& + \chi \left( \int_0^\epsilon  |e^{s-\epsilon}\xi|^2  \exp\left( s-\epsilon +  \dfrac{e^{2(s-\epsilon)} - 1}2|\xi|^2\right)\, ds\right)\sup_{s\in (0,\epsilon)} F(t+s)
\end{align*}
We deduce the following contraction estimate,
\begin{equation*} \label{eq:contraction self-similar}
\limsup_{\epsilon\to 0^+} \dfrac{F(t+\epsilon) - F(t)}{\epsilon} \leq -F(t) +  (\chi - 1)\left(\liminf_{\epsilon\to0^+}|\xi^*(t+\epsilon)|^2\right)  F(t)\, .
\end{equation*}
Hence the one-dimensional Keller-Segel equation \eqref{eq:1DKS Fourier} is a contraction with rate 1 with respect to the Fourier distance $d_1$.
\end{proof}


\subsection*{Acknowledgments}
Both authors are supported by the bilateral France-Spain project
FR2009-0019. JAC is partially supported by the projects
MTM2008-06349-C03-03 DGI-MCI (Spain) and 2009-SGR-345 from
AGAUR-Generalitat de Catalunya. Both authors thank CRM (Centre de
Recerca Matem\'atica) in Barcelona where part of this work was
done during the stay of the first author as a part of the thematic
program in Mathematical Biology in 2009.



\end{document}